\documentclass{amsart}
\usepackage{amsmath,amssymb,amsthm,amscd}
\usepackage{graphicx}
\theoremstyle{definition}
\newtheorem{definition}{Definition}
\theoremstyle{plain}
\newtheorem{theorem}{Theorem}
\newtheorem{lemma}{Lemma}
\newtheorem{corollary}{Corollary}
\DeclareMathOperator{\real}{Re}
\DeclareMathOperator{\Img}{Im}

\DeclareMathOperator{\SOn}{SO}
\DeclareMathOperator{\Un}{U}

\DeclareMathOperator{\Spn}{Sp}

\allowdisplaybreaks

\begin{document}

\title[Twistor lifts and factorization for conformal maps]
 {Twistor lifts and factorization for conformal maps from a surface to the Euclidean four-space}

\author{Kazuyuki Hasegawa}

\address{%
Faculty of Teacher Education, 
Institute of Human and Social Sciences, 
Kanazawa University, 
Kakuma-machi, Kanazawa, Ishikawa, 
920-1192, Japan}

\email{kazuhase@staff.kanazawa-u.ac.jp}

\author{Katsuhiro Moriya}
\address{
Division of Mathematics, 
Faculty of Pure and Applied Sciences, 
University of Tsukuba, 
1-1-1 Tennodai, Tsukuba, Ibaraki, 
305-8571, Japan}
\email{moriya@math.tsukuba.ac.jp}
\thanks{
This work was supported by JSPS KAKENHI Grant Number 25400063
and JSPS KAKENHI Grant Number 23540081.
}



\maketitle
\begin{abstract}
A conformal map from a Riemann surface to the Euclidean four-space is explained in terms of its twistor lift. 
A local factorization of a differential of a conformal map is obtained. 
As an application, the factorization of a differential provides an upper bound of 
the area of a super-conformal map around a branch point. 
\end{abstract}
\section{Introduction}\label{section:intro}
In classical surface theory, 
we consider an oriented surface to be the image of  
an isometric immersion from a two-dimensional oriented Riemannian manifold. 
To investigate the Riemannian geometric properties, 
we frequently employ  
an orthogonal complex structure that is compatible with the orientation of a two-dimensional Riemannian manifold. 
We employ theory of holomorphic functions, Riemann surfaces and holomorphic vector bundles. 
This method is successful and has been investigated in various studies.  For example, several important examples of minimal surfaces in Euclidean space are constructed 
by a meromorphic function and a holomorphic one-form on a Riemann surface by 
the Weierstrass representation formula \cite{Enneper64}, \cite{Weierstrass66}. 
The Hopf's theorem for constant mean curvature surfaces is 
proven by the holomorphic Hopf differential \cite{Hopf83}. 
A holomorphic function is a (branched) conformal immersion and 
the theory of holomorphic functions is a successful theory. 
We obtain an idea for constructing a theory of 
conformal immersions so that it includes the theory of holomorphic functions. 

The paper \cite{PP98} seems to be one of initial significant achievements using by 
this idea. They refer to a branched conformal immersion from a Riemann surface 
to the four-dimensional Euclidean space $\mathbb{E}^4$ as a \textit{conformal map}. 
A conformal map is considered to be a holomorphic map 
from a Riemann surface to the four-dimensional Euclidean space 
with respect to an almost complex structure along $f$. 
The subsequent papers show that this approach is 
fruitful. For example, \cite{FLPP01} introduces quaternionic holomorphic curves, which include holomorphic curves 
in complex projective space and 
obtains theorems that hold for holomorphic curves in complex projective space as special cases. 

The almost complex structure along a conformal map 
is considered to be a map from a Riemann surface to the twistor space of $\mathbb{E}^4$. 
A twistor lift is a pair that consists of a conformal map and an almost complex structure along the conformal map.  
An almost complex structure is invariant under conformal transforms 
of $\mathbb{E}^4$. 
We can expect that the twistor space is useful for studying 
conformal maps. 
We implement this idea in this paper. 

The twistor theory serves an important role in the study
of surfaces in four-dimensional Riemannian manifolds, in particular,
minimal surfaces (see \cite{Bryant82} and \cite{Friedrich84}, for example). 
Quaternionic holomorphic 
geometry \cite{BFLPP02} is another useful theory for studying the surfaces in the special case. 
As shown in Section \ref{sec:QH}, there is a close relation between 
the theory of twistor lifts and quaternionic holomorphic geometry. 
The use of twistor lifts has the advantage that
it induces a factorization of the differential of a conformal map into 
a factor which describes intrinsic geometry of a surface and other 
factors which describe extrinsic geometry of a surface. 
More precisely, in Section \ref{sec:conformal}, we show that the differential of a conformal map is 
factored by two maps into $\Spn(1)$ and 
a $(1,0)$-form locally. We refer to it as a \textit{canonical factorization}. 
We note that the $(1,0)$-form gives the intrinsic Riemannian invariant of a conformal map. 
The maps into Sp(1) give the generalized Gauss map of a surface. 

Our approach is motivated by researches of spinor structures 
for surfaces in three or four-space. 
If the ambient space is $\mathbb{E}^3$, 
Theorem 3.1.1 in \cite{KNPP2002} provides a canonical way for determining 
a spinor structure of $f^{\ast} T\mathbb{E}^3$
for a given conformal immersion $f$. 
However, this way does not work for immersions into $\mathbb{E}^4$.  
In \cite{BLR2013}, fixing a spinor structure of the tangent bundle of 
a Riemann surface and that of the normal bundle, a representation formula 
for immersions into four-dimensional space form is obtained.   
In this paper, we fix 
a spinor structure of the tangent bundle of $\mathbb{E}^4$. 
Therefore we begin our discussion with recalling the
twistor space of $\mathbb{E}^4$ after Salamon \cite{Salamon82} through 
the spinor structure. This also leads that 
the relation between the theory of twistor lifts and quaternionic holomorphic geometry 
is clarified. 

In Section \ref{sec:conformal}, we define a conformal map 
using a map from a Riemann surface to the twistor space and 
explain that this definition coincides with the definition of 
a conformal map in \cite{PP98}. 
Among twistor lifts of conformal maps, 
we distinguish a special lift that we refer to as a \textit{canonical lift}. 
The canonical lift induces the canonical factorization of 
a differential of a conformal map. 
In section \ref{sec:lcm}, we have a relation between 
the area of a conformal map and its canonical lifts. 
In Section \ref{sec:QH}, the relation between the theory of twistor lifts 
and quaternionic holomorphic geometry is given. In the last section, 
we give an application of the the canonical factorization to 
super-conformal maps.

The authors would like to thank the referees for their valuable comments to 
improve the presentation of this paper. 

\section{Preliminaries}\label{sec:prel}
Throughout this paper, all manifolds and maps are assumed to be smooth.
We review the twistor space of $\mathbb{E}^4$ after Salamon \cite{Salamon82}. 
\subsection{Elementary representation theory}
Let $V$ be a real four-dimensional vector space and let $\langle\enskip,\enskip\rangle$ be an inner product on $V$. 
We denote the norm of $v\in V$ by $|v|$. 
Let $(J_1,J_2,J_3)$ with $J_1\circ J_2=J_3$ be a hypercomplex structure of $V$ such that 
$\langle\enskip,\enskip\rangle$ is Hermitian with respect to $(J_1,J_2,J_3)$. 

We consider $V$ to be  
a right quaternionic module by 
\begin{gather*}
v(a_0+a_1i+a_2j+a_3k)=va_0-(J_1v)a_1-(J_2v)a_2-(J_3v)a_3
\end{gather*}
for $v\in V$ and $a_0$, $a_1$, $a_2$, $a_3\in\mathbb{R}$. 

Fix $v_0\in V$ with $|v_0|=1$. 
Define quaternionic linear automorphisms $\tilde{J}_1$, $\tilde{J}_2$, and  $\tilde{J}_3$ of $V$ by 
$\tilde{J}_nv_0=-J_nv_0$ $(n=1,2,3)$. 
Then $(\tilde{J}_1,\tilde{J}_2,\tilde{J}_3)$ is a hypercomplex structure of $V$ with $\tilde{J}_1\circ \tilde{J}_2=\tilde{J}_3$ such that 
$\langle\enskip,\enskip\rangle$ is Hermitian with respect to $(\tilde{J}_1,\tilde{J}_2,\tilde{J}_3)$. 
We consider $V$ to be  
a left quaternionic module by 
\begin{gather*}
(a_0+a_1i+a_2j+a_3k)v=a_0v+a_1\tilde{J}_1v+a_2\tilde{J}_2v+a_3\tilde{J}_3v.
\end{gather*}

We note that $iv_0=v_0i$, $jv_0=v_0j$ and $kv_0=v_0k$. 
Then, $V$ is isomorphic to the non-commutative associative algebra of 
all quaternions $\mathbb{H}$. We often identify $V$ with $\mathbb{H}$ in this manner. 
The vector $v_0\lambda=\lambda v_0\in V$ with $\lambda\in\mathbb{H}$ is identified with $\lambda\in\mathbb{H}$. 
The set $U=\{v_0\lambda:\lambda\in\mathbb{C}\}$ is identified with the set of all complex numbers $\mathbb{C}$. 

We obtain an orthogonal decomposition of $V$ by real vector spaces 
\begin{gather*}
V=V_c\oplus V_c^\perp,\enskip 
V_c=\{v_0r:r\in\mathbb{R}\}. 
\end{gather*}
Then, $V_c$ is identified with the set $\real\mathbb{H}$ of all real parts of quaternions and $V_c^\perp$ is identified with the set $\Img\mathbb{H}$ of all imaginary parts of quaternions. 
We denote the quaternionic conjugate of $v\in V\cong\mathbb{H}$ by $\overline{v}$. 

If we consider $V$ to be a right complex vector space with the complex structure $-J_1$, then we denote it by $V_+$. We obtain $V_+=U\oplus kU$. 
If we consider $V$ to be as a left complex vector space with complex structure $-\tilde{J}_1$, then we denote it by $V_-$. We obtain $V_-=U\oplus Uj$.  

For any $v\in V$ with $|v|=1$, 
a quadruplet $(v,-J_1v,-J_3v,-J_2v)$ is an orthonormal basis of $V$. 
The ordered orthonormal basis 
\begin{gather*}
\widetilde{v_0}:=(v_0,-J_1v_0,-J_3v_0,-J_2v_0)
\end{gather*}
determines an orientation. When we identify 
$V$ with $\mathbb{H}$, the ordered basis $\widetilde{v_0}$ is 
identified with $(1,i,k,j)$. 
The set of all orthonormal ordered bases $(v_1,v_2,v_3,v_4)$ with the same orientation 
as $\widetilde{v_0}$ constitutes the special orthogonal group $\SOn(4)$ by the relation 
\begin{gather*}
(v_1,v_2,v_3,v_4)=(v_0,-J_1v_0,-J_3v_0,-J_2v_0)\beta,\enskip \beta\in\SOn(4).
\end{gather*}

Let $I$ be an orthogonal complex structure of $V$. 
The set of all orthonormal ordered bases of the form $(v_1, -Iv_1,v_2,-Iv_2)$ with the same orientation as $\widetilde{v_0}$ constitutes a subgroup of $\SOn(4)$, which is 
isomorphic to the unitary group $\Un(2)$ by
\begin{gather*}
(v_1, -Iv_1,v_2,-Iv_2)=(v_0,-J_1v_0,-J_3v_0,-J_2v_0)\beta,\enskip \beta\in\Un(2).
\end{gather*}

The set of all orthonormal ordered bases of the form $(v_1, -J_1v_1,-J_3v_0,-J_2v_0)$ with the same orientation as $\widetilde{v_0}$ constitutes a subgroup of $\Un(2)$, which is 
isomorphic to the unitary group $\Un(1)$ by
\begin{gather*}
(v_1, -J_1v_1,-J_3v_0,-J_2v_0)=(v_0,-J_1v_0,-J_3v_0,-J_2v_0)\beta,\enskip \beta\in\Un(1).
\end{gather*}
The group $\Un(1)$ is isomorphic to the set of all unit complex numbers. 
Similarly, 
the set of all orthonormal ordered bases of the form $(v_0, -J_1v_0,v_2,-J_1v_2)$ with the same orientation as $\widetilde{v_0}$ constitutes a subgroup of $\Un(2)$, which is 
isomorphic to the unitary group $\Un(1)$ by
\begin{gather*}
(v_0, -J_1v_0,v_2,-J_1v_2)=(v_0,-J_1v_0,-J_3v_0,-J_2v_0)\beta,\enskip \beta\in\Un(1).
\end{gather*}

The set of all orthonormal ordered bases of the form $(v,-J_1v,-J_3v,-J_2v)$ with the same orientation as $\widetilde{v_0}$ constitutes a subgroup of $\SOn(4)$, which is 
isomorphic to the symplectic group $\Spn(1)$ by 
\begin{gather*}
(v, -J_1v,-J_3v,-J_2v)=(v_0,-J_1v_0,-J_3v_0,-J_2v_0)\beta,\enskip \beta\in\Spn(1). 
\end{gather*}
The symplectic group $\Spn(1)$ is isomorphic to the group of all unit quaternions. 
A double-covering $\phi\colon \Spn(1)\times\Spn(1)\to\SOn(4)$ is defined by 
\begin{align*}
&(av_0b^{-1},a(-J_1v_0)b^{-1},a(-J_3v_0)b^{-1},a(-J_2v_0)b^{-1})\\
=&(v_0,-J_1v_0,-J_3v_0,-J_2v_0)\phi(a,b),\enskip (a,b)\in \Spn(1)\times\Spn(1).
\end{align*}
Because $\phi(a,a)$ preserves the decomposition $V_c\oplus V_c^\perp$ with orientation,  
the set of all matrices of the form $\phi(a,a)$ $(a\in\Spn(1))$ constitutes the subgroup of $\SOn(4)$, which is isomorphic to $\SOn(3)$. 
The map $\phi$ composed with the inclusion $a\mapsto (a,a)$ of $\Spn(1)$ into $\Spn(1)\times\Spn(1)$ 
is a double-covering $\Spn(1)\to\SOn(3)$. 

The maps $\phi|_{\Un(1)\times \Spn(1)}$  and $\phi|_{\Spn(1)\times \Un(1)}$ are 
double-coverings of $\Un(2)$. 
Selecting the double-covering $\phi|_{\Un(1)\times \Spn(1)}$, we obtain  
\begin{gather*}
\SOn(4)/\Un(2)\cong(\Spn(1)\times\Spn(1))/(\Un(1)\times\Spn(1))=\Spn(1)/\Un(1).
\end{gather*}
We fix a complex line $L=\{v_0\lambda:\lambda\in\mathbb{C}\}$ in $V_+$. 
Let $a=a_0+a_1i+a_2k+a_3j$ $(a_0,a_1,a_2,a_3\in\mathbb{R})$. 
Then, $aL=\{(v_0(a_0+a_1i)-(J_3v_0)(a_2+a_3i))\lambda:\lambda\in\mathbb{C}\}$ is a complex line. 
Let $(W_0,W_1)$ be a holomorphic coordinate of $V_+$ such that 
$V_+=\{v_0W_0-(J_3v_0)W_1:W_0,W_1\in\mathbb{C}\}$ and 
let $[W_0,W_1]$ be the homogeneous coordinate of $\mathbb{P}(V_+)$. 
Then, $aL=[a_0+a_1i,a_2+a_3i]$. 
For $a\in\Spn(1)$, 
we denote $a\Un(1)\in\Spn(1)/\Un(1)$ by $a^\flat$. 
The correspondence $a^\flat\mapsto aL$ for any $a\in \Spn(1)$ identifies 
$\Spn(1)/\Un(1)$ with $\mathbb{P}(V_+)$. 

Consider $\Spn(1)$ as the three-dimensional sphere 
$S^3=\{a\in\mathbb{H}:|a|=1\}$. 
Let $S^2$ be the two-dimensional sphere $\{a\in\Img\mathbb{H}:|a|=1\}$.  
We obtain the Hopf map $H\colon S^3\to S^2$, $H(a)=aia^{-1}$ of the Hopf fibration. 
The map $\Phi_+\colon \Spn(1)/\Un(1)\to S^2$ defined by 
$\Phi_+(a^\flat)=aia^{-1}$ identifies $\Spn(1)/\Un(1)$ with $S^2$.                                                                                                

There is a bijective map $I_+$ from $\Spn(1)/\Un(1)$ to the set of all orthogonal complex structures of $V$ 
such that 
\begin{gather*}
(v_1,-I_+(a^\flat)v_1,v_2,-I_+(a^\flat)v_2)=(v_0,-J_1v_0,-J_3v_0,-J_2v_0)\phi(a,b). 
\end{gather*}

Similarly, selecting the double-covering $\phi|_{\Un(1)\times \Spn(1)}$, we obtain  
\begin{gather*}
\SOn(4)/\Un(2)\cong(\Spn(1)\times\Spn(1))/(\Spn(1)\times\Un(1))=\Spn(1)/\Un(1).
\end{gather*}
For $b^{-1}\in \Spn(1)$, we denote $\Un(1) b^{-1}\in\Spn(1)/\Un(1)$ by $(b^{-1})^\sharp$. 
The correspondence $(b^{-1})^\sharp \mapsto Lb^{-1}$ for any $b^{-1}\in \Spn(1)$ identifies 
$\Spn(1)/\Un(1)$ with $\mathbb{P}(V_-)$. 
The map $\Phi_-\colon \Spn(1)/\Un(1)\to S^2$ defined by 
$\Phi_-((b^{-1})^\sharp)=bib^{-1}$ identifies $\Spn(1)/\Un(1)$ with $S^2$.                  
The bijective map $I_-$ from $\Spn(1)/\Un(1)$ to the set of all orthogonal complex structures of $V$ exists 
such that 
\begin{gather*}
(v_1,-I_-((b^{-1})^\sharp)v_1,v_2,-I_-((b^{-1})^\sharp)v_2)=(v_0,\tilde{J}_1v_0,\tilde{J}_3v_0,\tilde{J}_2v_0)\phi(a,b). 
\end{gather*}
Then, 
\begin{align*}
(v_1,-Iv_1,v_2,-Iv_2) &= (v_0,\tilde{J}_1v_0,\tilde{J}_3v_0,\tilde{J}_2v_0)\phi(a,b),\\
-I &= -I_+(a^\flat)=-I_-((b^{-1})^\sharp). 
\end{align*}

We note that 
\begin{gather*}
-I_-((b^{-1})^\sharp)v=v\,bib^{-1}, \enskip v\in V. 
\end{gather*}
For $\beta\in\Spn(1)/\Un(1)$ with $\beta=(b^{-1})^\sharp$, 
we exchange the notation $I_-(\beta)$ with $\mathcal{I}_{-}^{\beta}$:
\begin{gather*}
-\mathcal{I}_-^\beta v=v\Phi_-(\beta). 
\end{gather*}
For $\alpha\in\Spn(1)/\Un(1)$, define the orthogonal complex structure $\mathcal{I}^{\alpha}_+$ by 
\begin{gather*}
-\mathcal{I}^{\alpha}_+(v)=-\Phi_+(\alpha)v.
\end{gather*}
Then,  
\begin{gather*}
-\mathcal{I}^{\alpha}_+v_1=Iv_1,\enskip -\mathcal{I}^{\alpha}_+v_2=-Iv_2. 
\end{gather*}

Let $V_1$ be the subspace of $V$ spanned by $v_1$ and $-Iv_1$ and 
let $V_2$ be the subspace of $V$ spanned by $v_2$ and $-Iv_2$. 
Then, 
\begin{gather*}
V_1=\{v\in V:\mathcal{I}_+^\alpha v=-\mathcal{I}_-^\beta v\},\enskip 
V_2=\{v\in V:\mathcal{I}_+^\alpha v=\mathcal{I}_-^\beta v\},\\
V=V_1\oplus V_2.
\end{gather*} 
\subsection{Twistor space}\label{sec:twistor_sp}
Let $TV$ be the tangent bundle of $V$ and let $T_vV$ be the tangent space of $V$ at $v$. 
We identify $T_vV$ with $V$ in the usual manner. 
We denote the integrable hypercomplex structures and the Riemannian metric induced from $V$ by the same symbols: $(J_1,J_2,J_3)$, $(\tilde{J}_1,\tilde{J}_2,\tilde{J}_3)$ and $\langle\enskip,\enskip\rangle$ respectively. 
Then, $\langle\enskip,\enskip\rangle$ is 
Hermitian with respect to $(J_1,J_2,J_3)$ and $(\tilde{J}_1,\tilde{J}_2,\tilde{J}_3)$.

Let 
$\widetilde{A_0}=(A_0,-J_1A_0,-J_3A_0,-J_2A_0)$ be 
an orthonormal ordered frame that corresponds to $\widetilde{v_0}$. 
Then, $iA_0=A_0i$, $jA_0=A_0j$ and $kA_0=A_0k$. 
The set of all orthonormal ordered frames $(A_1,A_2,A_3,A_4)$ with the same orientation as $\widetilde{A_0}$ constitute a principal $\SOn(4)$-bundle $P$ over $V$. 
The set of all orthonormal ordered frames of the form $(A_1, -IA_1,A_2,-IA_2)$ 
with orthogonal almost complex structure $I$ of $TV$ and the same orientation as $\widetilde{A_0}$ 
constitutes a principal $\Un(2)$-bundle $Q$ over $V$. 
Then, $Q$ is identified with a section of the fiber bundle 
\begin{gather*}
\pi^V\colon Z\to V,\enskip 
Z=P\times_{\SOn(4)}\SOn(4)/\Un(2)=V\times \SOn(4)/\Un(2).
\end{gather*}
The bundle $Z$ is referred to as the \textit{twistor space} of $V$. 
The set of all sections of $\pi^V$ is considered to be the set of all 
almost complex structures of $V$. 
For an orthogonal almost complex structure $I$ of $V$, 
we obtain the orthonormal ordered frame $(A_1,-IA_1,A_2,-IA_2)$, which corresponds to 
the map $(a,b)\colon V\to \Spn(1)\times\Spn(1)$ 
by the equation 
\begin{gather*}
(A_1,-IA_1,A_2,-IA_2)=(A_0,-J_1A_0,-J_3A_0,-J_2A_0)\phi(a,b),\\
-I=-I_+(\alpha)=-I_-(\beta),\enskip 
\alpha=a^\flat,\enskip \beta=(b^{-1})^\sharp. 
\end{gather*}

As we stated in Section 1, 
we consider the twistor space of $\mathbb{E}^4$ through the spinor structure. 
we consider the spinor structure of $\mathbb{E}^4$ . 
Let $\widetilde{P}$ be the spinor structure of $P$. Then, $\tilde{P}$ is the $\Spn(1)\times\Spn(1)$-bundle, which is the lift of $P$ by $\phi$. 
Selecting $\phi|_{\Un(1)\times\Spn(1)}$ for the double covering of $\Un(2)$ and 
considering $V$ to be the right complex vector space $V_+$, 
the twistor space is identified with the fiber bundle 
\begin{align*}
&\tilde{\pi}^{V}_+\colon\tilde{Z}_+\to V,\\
\tilde{Z}_+ &=\widetilde{P}\times_{\Spn(1)\times \Spn(1)}(\Spn(1)\times \Spn(1))/(\Un(1)\times\Spn(1))\\
&= V\times \Spn(1)/\Un(1)\cong V\times \mathbb{P}(V_+).
\end{align*}

Let $J_{\mathbb{P}(V_+)}$ be the complex structure of $\mathbb{P}(V_+)$. 
An integrable complex structure $J_{\tilde{Z}_+}$ of $\tilde{Z}_+$ is defined by 
\begin{gather*}
J_{\tilde{Z}_+}(A,S)=(-I_+(\alpha)A,J_{\mathbb{P}(V_+)}S),\\
(v,\alpha)\in V\times\mathbb{P}(V_+),\enskip(A,S)\in T_{(v,\alpha)}(V \times \mathbb{P}(V_{+})) \cong T_{v}V \times T_{\alpha} \mathbb{P}(V_{+}). 
\end{gather*}

Selecting $\phi|_{\Spn(1)\times\Un(1)}$ for the double covering of $\Un(2)$ and 
considering $V$ to be the left complex vector space $V_-$, 
the twistor space is also identified with the fiber bundle 
\begin{align*}
&\tilde{\pi}^{V}_-\colon\tilde{Z}_-\to V,\\
\tilde{Z}_- &=\widetilde{P}\times_{\Spn(1)\times \Spn(1)}(\Spn(1)\times \Spn(1))/(\Spn(1)\times\Un(1))\\
&= V\times \Spn(1)/\Un(1)\cong V\times \mathbb{P}(V_-).
\end{align*}
Let $J_{\mathbb{P}(V_-)}$ be the complex structure of $\mathbb{P}(V_-)$. 
Using a similar discussion, 
we obtain the complex structure $J_{\tilde{Z}_-}$ of $\tilde{Z}_-$, which is defined by 
\begin{gather*}
J_{\tilde{Z}_-}(A,S)=(-I_-(\beta)A,J_{\mathbb{P}(V_-)}S),\\
(v,\beta)\in V\times\mathbb{P}(V_-),\enskip(A,S)\in T_{(v,\beta)}(V \times \mathbb{P}(V_{-})) \cong T_{v}V \times T_{\beta} \mathbb{P}(V_{-}). 
\end{gather*}
\section{Conformal maps}\label{sec:conformal}
We explain conformal maps by the twistor setting.
Let $\Sigma$ be a Riemann surface with the complex structure $J_\Sigma$. 
Recall that 
a holomorphic function on $\Sigma$ is a conformal map $h\colon \Sigma\to\mathbb{C}$ such that 
$dh\circ J_\Sigma=i\,dh=dh\,i$. 
We consider a conformal map from $\Sigma$ to $V$ as an analog of a holomorphic function. 
For the map $f\colon \Sigma\to V$, 
denote by $\Upsilon_f$ the set of all maps $\mu$ from $\Sigma$ such that 
$\mu(p)$ is an orthogonal complex structure of $T_{f(p)}V$ for each $p\in \Sigma$. 
\begin{definition}\label{def:conformal}
We refer to a non-constant map $f\colon \Sigma\to V$ a \textit{conformal map} if there exists a map 
$I^\Sigma\in\Upsilon_f$ such that 
$df\circ J_\Sigma=-I^\Sigma\,df$. 
\end{definition}
If a conformal map $f$ is not an immersion at $p$, then $df$ is the zero map at $p$ and $p$ is a branch point of $f$. 

Let $T\Sigma$ be the tangent bundle of $\Sigma$ 
and let $T_p\Sigma$ be the tangent space of $\Sigma$ at $p$. 
Then, the tangent space of $f$ at $p$ is $df(T_p\Sigma)$. 
Denote the normal space of $f$ at $p$ by 
$(df(T_p\Sigma))^\perp$. 
The twistor space of $V$ explains conformal maps as follows. 

\begin{theorem}\label{thm:conformal}
If $f\colon\Sigma\to V$ is a conformal map with $df\circ J_\Sigma=-I^\Sigma\,df$ for a map 
$I^\Sigma\in \Upsilon_f$, 
then maps $\alpha\colon\Sigma\to\Spn(1)/\Un(1)$ and $\beta\colon\Sigma\to\Spn(1)/\Un(1)$ exist 
such that 
\begin{gather*}
I^\Sigma=I_+(\alpha)=I_-(\beta). 
\end{gather*}
At each point $p$ of $\Sigma$, the open set $U$ including $p$, local lifts $a\colon U\to \Spn(1)$ and  $b\colon U\to \Spn(1)$ for $\alpha$ and $\beta$, respectively,  and 
a complex $(1,0)$-form $\eta$ on $U$ exist 
such that  
\begin{gather*}
df=ak\,\eta\,b^{-1}. 
\end{gather*}
\end{theorem}
\begin{proof}
We only have to prove the theorem for a point in which $f$ is an immersion. 

Assume that $f$ is an immersion at $p\in\Sigma$. 
Because $I^\Sigma$ is an orthogonal complex structure of $V$, 
the maps $\alpha\colon\Sigma\to\Spn(1)/\Un(1)$ and 
$\beta\colon\Sigma\to\Spn(1)/\Un(1)$ exist such that 
$I^\Sigma=I_+(\alpha)=I_-(\beta)$ per the discussion in Section \ref{sec:prel}. 
Because $df(T_p\Sigma)$ is preserved by $I^\Sigma(p)$, 
we may assume that the existence the open set $U$, including $p$, and 
the ordered orthogonal local frame of $TV$ on $U$ of the form 
\begin{gather*}
(A_1,-I^\Sigma A_1,A_2, -I^\Sigma  A_2)
\end{gather*}
such that $df(T\Sigma)$ is framed by $A_2$ and $-I^\Sigma  A_2$. 
The maps $a\colon U\to\Spn(1)$ and $b\colon U\to\Spn(1)$ exist such that 
\begin{gather*}
(A_1,-I^\Sigma A_1,A_2, -I^\Sigma A_2)=
(A_0,-J_1A_0,-J_3A_0, -J_2 v_0)\phi(a,b),\\
a^\flat=\alpha,\enskip (b^{-1})^\sharp=\beta.
\end{gather*}
Because 
$I^\Sigma=\mathcal{I}_+^\alpha=\mathcal{I}^\beta_-$ on $df(T_q\Sigma)$ for each $q\in U$, 
we obtain  
\begin{gather*}
df\circ J_\Sigma=-aia^{-1}\,df=df\,bib^{-1}. 
\end{gather*}
Because 
\begin{gather*}
a^{-1}\,(df\circ J_\Sigma)\,b=-ia^{-1}\,df\,b=a^{-1}\,df\,bi,
\end{gather*}
the complex $(1,0)$-form $\eta$ exists such that $a^{-1}\,df\,b=k\,\eta$.  
Therefore, $df=ak\,\eta\,b^{-1}$. 
\end{proof}

\begin{center}
\vspace{-1mm}
\hspace{-15mm}
{\unitlength 0.1in%
\begin{picture}(23.8700,10.4000)(5.0000,-16.1000)%
\put(16.6000,-14.6000){\makebox(0,0)[lb]{$\Sigma$}}%
%
\special{pn 8}%
\special{pa 1800 1417}%
\special{pa 2800 1417}%
\special{fp}%
\special{sh 1}%
\special{pa 2800 1417}%
\special{pa 2733 1397}%
\special{pa 2747 1417}%
\special{pa 2733 1437}%
\special{pa 2800 1417}%
\special{fp}%
%
\special{pn 8}%
\special{pa 1730 1250}%
\special{pa 1730 850}%
\special{fp}%
\special{sh 1}%
\special{pa 1730 850}%
\special{pa 1710 917}%
\special{pa 1730 903}%
\special{pa 1750 917}%
\special{pa 1730 850}%
\special{fp}%
\put(28.2000,-14.6000){\makebox(0,0)[lb]{${\mathbb R}^{4}$}}%
\put(27.7000,-8.3000){\makebox(0,0)[lb]{$\tilde{Z}_{+}={\mathbb R}^{4} \times ({\rm Sp}(1)/{\rm U}(1))$}}%
\put(14.1000,-8.3000){\makebox(0,0)[lb]{${\rm Sp}(1)/{\rm U}(1)$}}%
\put(6.5000,-14.8000){\makebox(0,0)[lb]{$S^{2}$}}%
\put(5.0000,-8.3000){\makebox(0,0)[lb]{${\rm Sp}(1)$}}%
%
\special{pn 8}%
\special{pa 760 880}%
\special{pa 760 1320}%
\special{fp}%
\special{sh 1}%
\special{pa 760 1320}%
\special{pa 780 1253}%
\special{pa 760 1267}%
\special{pa 740 1253}%
\special{pa 760 1320}%
\special{fp}%
\special{pa 760 1320}%
\special{pa 760 1320}%
\special{fp}%
%
\special{pn 8}%
\special{pa 970 750}%
\special{pa 1370 750}%
\special{fp}%
\special{sh 1}%
\special{pa 1370 750}%
\special{pa 1303 730}%
\special{pa 1317 750}%
\special{pa 1303 770}%
\special{pa 1370 750}%
\special{fp}%
%
\special{pn 8}%
\special{pa 2700 760}%
\special{pa 2100 760}%
\special{fp}%
\special{sh 1}%
\special{pa 2100 760}%
\special{pa 2167 780}%
\special{pa 2153 760}%
\special{pa 2167 740}%
\special{pa 2100 760}%
\special{fp}%
%
\special{pn 8}%
\special{pa 2860 860}%
\special{pa 2860 1300}%
\special{fp}%
\special{sh 1}%
\special{pa 2860 1300}%
\special{pa 2880 1233}%
\special{pa 2860 1247}%
\special{pa 2840 1233}%
\special{pa 2860 1300}%
\special{fp}%
\special{pa 2860 1300}%
\special{pa 2860 1300}%
\special{fp}%
%
\special{pn 8}%
\special{pa 1640 1437}%
\special{pa 1640 1437}%
\special{fp}%
\special{pa 840 1437}%
\special{pa 840 1437}%
\special{fp}%
%
\special{pn 8}%
\special{pa 1640 1420}%
\special{pa 840 1420}%
\special{fp}%
\special{sh 1}%
\special{pa 840 1420}%
\special{pa 907 1440}%
\special{pa 893 1420}%
\special{pa 907 1400}%
\special{pa 840 1420}%
\special{fp}%
\put(22.7000,-16.1000){\makebox(0,0)[lb]{$f$}}%
%
\special{pn 8}%
\special{pa 1920 1280}%
\special{pa 2770 850}%
\special{fp}%
\special{sh 1}%
\special{pa 2770 850}%
\special{pa 2701 862}%
\special{pa 2722 874}%
\special{pa 2720 898}%
\special{pa 2770 850}%
\special{fp}%
\special{pa 2770 850}%
\special{pa 2770 850}%
\special{fp}%
\put(24.1000,-12.4000){\makebox(0,0)[lb]{$\tilde{f}_{+}$}}%
\put(15.6000,-11.4000){\makebox(0,0)[lb]{$\alpha$}}%
\put(10.9000,-7.0000){\makebox(0,0)[lb]{$\flat$}}%
\put(5.5000,-11.6000){\makebox(0,0)[lb]{$H$}}%
\put(7.7000,-16.1000){\makebox(0,0)[lb]{$(N=)-aia^{-1}$}}%
%
\special{pn 8}%
\special{pa 1580 1310}%
\special{pa 1580 1310}%
\special{fp}%
\special{pa 1580 1310}%
\special{pa 940 850}%
\special{fp}%
\special{sh 1}%
\special{pa 940 850}%
\special{pa 982 905}%
\special{pa 983 881}%
\special{pa 1006 873}%
\special{pa 940 850}%
\special{fp}%
\special{pa 940 850}%
\special{pa 940 850}%
\special{fp}%
\put(11.0000,-11.9000){\makebox(0,0)[lb]{$a$}}%
\end{picture}}%
\\
\vspace{2mm}
\vspace{0.2cm}
\hspace{-3mm}Fig 1 : conformal map and twistor space
\end{center}

Because $\phi(a,a)$ preserves $V_c^\perp$ for each $a$, we 
immediately obtain a three-dimensional version of Theorem \ref{thm:conformal}. 
\begin{corollary}\label{cor:conformal3}
If $f\colon\Sigma\to V_c^\perp$ is a conformal map with $df\circ J_\Sigma=-I^\Sigma\,df$ for the map 
$I^\Sigma\in \Upsilon_f$, 
the maps $\alpha\colon\Sigma\to\Spn(1)/\Un(1)$ and $\beta\colon\Sigma\to\Spn(1)/\Un(1)$ 
exist such that 
\begin{gather*}
I^\Sigma=I_+(\alpha)=I_-(\beta). 
\end{gather*}
At each point $p$ of $\Sigma$, the open set $U$, including $p$, a map $a\colon U\to \Spn(1)$ with $a^\flat=\alpha$ and $(a^{-1})^\sharp=\beta$ and 
the complex $(1,0)$-form $\eta$ on $U$ 
exist 
such that  
\begin{gather*}
df=ak\,\eta\,a^{-1}. 
\end{gather*}
\end{corollary}

We review a holomorphic map by this formulation. 
A holomorphic map $f\colon \Sigma\to V_+$ is a conformal map 
with $d f\circ J_\Sigma=-J_1\,d f=df\,i$ and 
a holomorphic map $f\colon \Sigma\to V_-$ is a conformal map 
with $d f\circ J_\Sigma=-\tilde{J}_1\,df=-i\,d f$. 

An orthogonal almost complex structure of $V$ is preserved by 
a conformal transformation of $V$. 
Thus to analyze a conformal map by its lift to the twistor space is 
a natural idea. 
We distinguish the following lifts:
\begin{definition}\label{def:clift}
Let $f\colon\Sigma\to V$ be a conformal map with $df\circ J_\Sigma=-I^\Sigma\,df$ and $\alpha\colon \Sigma\to \Spn(1)/\Un(1)$ and $\beta\colon \Sigma\to \Spn(1)/\Un(1)$ be maps with $I^\Sigma=I_+(\alpha)=I_-(\beta)$. 
We refer to $\tilde{f}_+=(f,\alpha)\colon \Sigma\to \tilde{Z}_+$ as a \textit{left canonical lift} of $f\colon \Sigma\to V$ 
and $\tilde{f}_-=(f,\beta)\colon \Sigma\to \tilde{Z}_-$ as a \textit{right canonical lift} of $f\colon \Sigma\to V$. 
\end{definition}
A left or right canonical lift is referred to a twistor lift in \cite{CU97} and \cite{Friedrich84}. 

\begin{definition}\label{def:cf}
Let $f\colon \Sigma \to V$ be  a conformal map  with left canonical lift $(f,\alpha)$ and 
right canonical lift $(f,\beta)$.  
Assume that $a\colon \Sigma\to \Spn(1)$ and $b\colon \Sigma\to \Spn(1)$ are maps with $a^\flat=\alpha$ and 
$(b^{-1})^\sharp=\beta$.  
We refer to $df=ak\,\eta\,b^{-1}$ 
with the complex $(1,0)$-form $\eta$ on $\Sigma$ as 
a \textit{canonical factorization} of $df$ by $a$, $b^{-1}$ and $\eta$. 
\end{definition}
\section{Local conformal maps}\label{sec:lcm}
In this section, we investigate properties of local conformal maps by a canonical factorization. 

At first, we assume that $\Sigma$ is a simply-connected open subset of $\mathbb{C}$. 
We denote the standard holomorphic coordinate of $\mathbb{C}$ by $z$. 
Then, a $(1,0)$-form is $c\,dz$ for a complex function $c$. 
Then, $\eta=-ka^{-1}\,df\,b=c\,dz$ for a complex function $c$. 
Theorem \ref{thm:conformal} delivers a method of construction for a conformal map.  
\begin{lemma}\label{lem:diff}
If the maps $a$, $b\colon\Sigma\to \Spn(1)$ and 
the complex $(1,0)$-form $\eta$ on $\Sigma$ satisfies 
\begin{gather}
da\wedge k\,\eta\,b^{-1}+ak\,d\eta\,b^{-1}-ak\,\eta\wedge db^{-1}=0, \label{eq:integrability}
\end{gather}
a conformal map $f\colon \Sigma\to V$ exists with 
$df=ak\,\eta\,b^{-1}$. 
\end{lemma}
\begin{proof}
Differentiating the one-form $ak\,\eta\,b^{-1}$, we obtain  
\begin{gather*}
d(ak\,\eta\,b^{-1})=da\wedge k\,\eta\,b^{-1}+ak\,d\eta\,b^{-1}-ak\,\eta\wedge db^{-1}.
\end{gather*}
If the maps $a$, $b\colon\Sigma\to \Spn(1)$ and 
the complex $(1,0)$-form $\eta$ satisfies 
the equation \eqref{eq:integrability}, 
then the map $f\colon \Sigma\to V$ exists with 
$df=ak\,\eta\,b^{-1}$. 
Because $df\circ J_\Sigma=-aia^{-1}\,df=df\,bib^{-1}$, the map $f$ is conformal. 
\end{proof}
In the following section, we assume that the conformal map $f$ has a canonical factorization $df=ak\,\eta\,b^{-1}$. 
The maps $a$, $b^{-1}$ and the one-form $\eta$ of a canonical factorization 
$df=ak\,\eta\,b^{-1}$ are not uniquely determined. 
If $u$ and $v$ are maps from $\Sigma$ to $\Un(1)$, then 
$(au)^\flat=a^\flat$ and $(vb^{-1})^\sharp=(b^{-1})^\sharp$.  
Then, 
\begin{gather*}
df=ak\,\eta\,b^{-1}=(au)(u^{-1}k\,\eta\,v^{-1})(bv^{-1})^{-1}
=(au)k(u\,\eta\,v^{-1})(bv^{-1})^{-1}. 
\end{gather*}
Let $\Omega^{(1,0)}$ be the set of all complex one-forms of type $(1,0)$ on $\Sigma$. Then, $\Un(1)$ acts on $\Omega^{(1,0)}$ by multiplication. 
For a conformal map $f$ with canonical factorization $df=ak\,\eta\,b^{-1}$, 
we obtain the unique triplet $(a^\flat,(b^{-1})^\sharp,[\eta])$, 
which consisting of $a^\flat$, $(b^{-1})^\sharp\colon \Sigma\to\Spn(1)/\Un(1)$ and 
$[\eta]\in\Omega^{(1,0)}/\Un(1)$. 

By Lemma \ref{lem:diff}, we obtain a representation formula for a conformal map $f\colon \Sigma\to V$ with the canonical factorization $df=ak\,\eta\,b^{-1}$: 
\begin{gather*}
f(p)=\int_{\gamma} ak\,\eta\,b^{-1}+f(p_0),\\
da\wedge k\,\eta\,b^{-1}+ak\,d\eta\,b^{-1}-ak\,\eta\wedge db^{-1}=0, \\
a,\, b\colon\Sigma\to \Spn(1),\enskip 
\eta\circ J_\Sigma=i\,\eta=\eta\,i.
\end{gather*}
Here, $\gamma$ is a path from $p_0$ to $p$.  
The zeros of $\eta$ are the branch points of $f$. 

We fix a canonical factorization $df=ak\,\eta\,b^{-1}$. 
If a $(1,0)$-form $\eta$ is nowhere vanishing, then 
$\eta$ is a global section of a real line bundle 
$l(\eta)=\cup_{p \in \Sigma} \{r\eta_p:r\in\mathbb{R}\}$ with the projection 
$\pi_{l(\eta)}\colon l(\eta)\to\Sigma$, $\pi_{l(\eta)}(r\eta_p)=p$.
\begin{lemma}\label{lem:cinv}
Let $f\colon \Sigma\to V$ be a conformal immersion with 
canonical factorization $df=ak\,\eta\,b^{-1}$. 
Let $\tilde{f}$ be an orientation-preserving conformal transform of $f$ in $V$. 
Then, the canonical factorization $d\tilde{f}=\tilde{a}k\,\tilde{\eta}\,\tilde{b}^{-1}$ exists such that $\tilde{a}^{-1}\,d\tilde{a}=a^{-1}\,da$, $\tilde{b}^{-1}\,d\tilde{b}=b^{-1}\,db$ and 
$l(\eta)=l(\tilde{\eta})$. 

If $\tilde{f}$ is a Euclidean motion of $f$, then 
the canonical factorization $d\tilde{f}=\tilde{a}k\,\tilde{\eta}\,\tilde{b}^{-1}$ exists with $\tilde{\eta}=\eta$. 
\end{lemma}
\begin{proof}
The map
\begin{gather*}
\tilde{f}=\lambda f\mu^{-1}+\nu \enskip (\lambda,\mu\in\mathbb{H}\setminus\{0\},\, \nu\in V) 
\end{gather*}
is a conformal transform of $\tilde{f}$. 
The differential of $\tilde{f}$ is 
\begin{gather*}
d\tilde{f}=\lambda\,df\,\mu^{-1}=\lambda ak\,\eta\,b^{-1}\mu^{-1}
=\frac{\lambda }{|\lambda|}ak\,\frac{|\lambda|}{|\mu|}\,\eta\, b^{-1}\frac{\mu ^{-1}}{|\mu|^{-1}},
\end{gather*}
Thus the canonical factorization $d\tilde{f}=\tilde{a}k\,\tilde{\eta}\,\tilde{b}^{-1}$ with 
\begin{gather*}
\tilde{a}=\frac{\lambda }{|\lambda|}a,\enskip 
\tilde{b}=\frac{\mu}{|\mu|} b,\enskip 
\tilde{\eta}=\frac{|\lambda|}{|\mu|}\,\eta
\end{gather*}
satisfies $\tilde{a}^{-1}\,d\tilde{a}=a^{-1}\,da$, $\tilde{b}^{-1}\,d\tilde{b}=b^{-1}\,db$ and 
$l(\tilde{\eta})=l(\eta)$. 

A Euclidean motion $\tilde{f}$ of $f$ is 
\begin{gather*}
\tilde{f}=\lambda f\mu^{-1}+\nu\enskip (\lambda,\mu\in\Spn(1),\nu\in V).
\end{gather*}
Then, we obtain the factorization 
\begin{gather*}
d\tilde{f}=\lambda\,df\,\mu^{-1}=(\lambda a)k\,\eta\,(\mu b)^{-1}. 
\end{gather*}
Because $|\lambda a|=|\mu b|=1$, this result is a
canonical factorization with $\tilde{\eta}=\eta$. 
\end{proof}

Because we can fix the complex $(1,0)$-from $\eta$ under 
Euclidean motions, 
the one-form $\eta$ includes Riemannian geometric information of $f$. 
Because the first fundamental form of $f$ is 
\begin{gather*}
\frac{1}{2}(df\otimes_{\mathbb{R}} d\overline{f}+d\overline{f}\otimes_{\mathbb{R}}df)
=\frac{1}{2}(a\,\eta\otimes_{\mathbb{R}}\overline{\eta}\,a^{-1}+b\,\overline{\eta}\otimes_{\mathbb{R}}\eta \,b^{-1}),
\end{gather*}
the $(1,0)$-one-form $\eta$ can generally describe a part of the Riemannian geometric properties. 
If the image of $f$ is included in $V_c^\perp$, then 
the $(1,0)$ form completely explains the Riemannian geometric properties.  
\begin{lemma}\label{ref:ff}
If $df=ak\,\eta\,b^{-1}$ is a canonical factorization of 
the conformal map $f\colon \Sigma\to V_c^\perp$, then 
the first fundamental form is 
\begin{gather*}
\frac{1}{2}(\eta\otimes_{\mathbb{R}}\overline{\eta}+\overline{\eta}\otimes_{\mathbb{R}}\eta). 
\end{gather*}
\end{lemma}
\begin{proof}
If the codomain of $f$ is contained in $V_c^\perp$, we 
may assume that $a=b$. 
Then, the first fundamental form is 
\begin{align*}
\frac{1}{2}(a\,\eta\otimes_{\mathbb{R}}\overline{\eta}\,a^{-1}+a\,\overline{\eta}\otimes_{\mathbb{R}}\eta \,a^{-1})
&=\frac{1}{2}a(\eta\otimes_{\mathbb{R}}\overline{\eta}+\overline{\eta}\otimes_{\mathbb{R}}\eta)a^{-1}\\
&=\frac{1}{2}(\eta\otimes_{\mathbb{R}}\overline{\eta}+\overline{\eta}\otimes_{\mathbb{R}}\eta). 
\end{align*}
\end{proof}

If the codomain of $f$ is not contained in $V_c^\perp$, then $\eta$ 
is insufficient for explaining the Riemannian geometric properties of $f$. 
However, we observe that the area of $f$ is described by $\eta$ 
as follows: 

We denote the $L^2$-norm of a one-form $\omega$ by $\|\omega\|_\Sigma$:
\begin{gather*}
\|\omega\|_\Sigma=\left(-\int_\Sigma\omega\wedge(\overline{\omega}\circ J_\Sigma)\right)^{1/2}.
\end{gather*}
In the space of all square integrable one-forms, an inner product is defined as 
\begin{gather*}
\langle\!\langle \omega_1,\omega_2\rangle\!\rangle_\Sigma
=-\frac{1}{2}\int_\Sigma(\omega_1\wedge \overline{\omega_2}\circ J_\Sigma+\omega_2\wedge \overline{\omega_1}\circ J_\Sigma). 
\end{gather*}

For the conformal map $f\colon \Sigma\to V$, we denote the area element of $f$ by $dA$ and 
denote the area of $f$ by $A(f)$. 
Let $z=x+iy$ be a local holomorphic coordinate of $\Sigma$ such that  $(x,y)$ is a local real coordinate. 
Then, 
\begin{gather*}
dA=\sqrt{|f_x|^2|f_y|^2-\langle f_x,f_y\rangle}\,dx\wedge dy=|f_x||f_y|\,dx\wedge dy
=-\frac{1}{2}\,df\wedge (d\overline{f}\circ J_\Sigma),\\
2A(f)=\|df\|^2_\Sigma. 
\end{gather*}

We recall quaternionic holomorphic geometry (see also Corollary \ref{lem:PP} in the next section). 
Assume that there exist maps 
$N$, $\tilde{N}\colon \Sigma\to\Img\mathbb{H}\cap\Spn(1)$ such that 
$df\circ J_{\Sigma}=N\,df=-df\,\tilde{N}$. 
Then 
\begin{gather*}
dA=\frac{1}{2}\,df\wedge d\overline{f}\,N=-\frac{1}{2}\,df\wedge \tilde{N}\,d\overline{f}=\frac{1}{2}N\,df\wedge d\overline{f}.
\end{gather*}
We see that the area involves the maps $N$ and $\tilde{N}$. 
The maps  $N$ and $\tilde{N}$ are written locally as 
$N=-aia^{-1}$ and $\tilde{N}=-bib^{-1}$. 
Hence it is natural to expect that 
the maps $a$ and $b$ are involved in the area.
However, we have the following formula for the area 
which does not includes $a$ and $b$.
\begin{lemma}\label{lem:area}
Let $f\colon\Sigma\to V$ be a conformal map with the canonical factorization 
$df=ak\,\eta\,b^{-1}$. Then,  
\begin{gather*}
2A(f)=\|\eta\|_\Sigma^2.
\end{gather*}
\end{lemma}
\begin{proof}
The area element of $f$ is 
\begin{gather*}
dA=-\frac{1}{2}ak\,\eta\,b^{-1}\wedge(\overline{ak\,\eta\,b^{-1}}\circ J_\Sigma)
=-\frac{1}{2}\eta\wedge(\overline{\eta}\circ J_\Sigma).
\end{gather*}
Thus, the lemma holds. 
\end{proof}

\section{Quaternionic holomorphic geometry}\label{sec:QH}
We collect the relation between twistor lifts and quaternionic holomorphic geometry. 
We have a relation among the area of a conformal map, that of its Darboux transform and 
that of its canonical lift. 
Considering the spinor structure in Section \ref{sec:twistor_sp}, 
we identify the twistor space of $V$ with $\tilde{Z}_{\pm}$. 
We arrive to the definition of conformal maps in Pedit and Pinkall \cite{PP98}
by Theorem~\ref{thm:conformal} (see also Fig 1):
\begin{corollary}\label{lem:PP}
A non-constant map $f\colon \Sigma\to V$ is a conformal map 
if and only if 
the maps $N$, $\tilde{N}\colon \Sigma\to\Img\mathbb{H}\cap\Spn(1)$ 
exist such that 
$df\circ J_\Sigma=N\,d f=-df\,\tilde{N}$. 
\end{corollary}
\begin{proof}
Assume that $f\colon\Sigma\to V$ is a conformal map with $df\circ J_\Sigma=-I^\Sigma\,df=-I_+(\alpha)\,df=-I_-(\beta)\,df$. 
Let $N=-\Phi_+(\alpha)$ and $\tilde{N}=-\Phi_-(\beta)$. 
By Theorem \ref{thm:conformal}, we obtain $df\circ J_\Sigma=N\,df=-df\,\tilde{N}$. 

For the maps $N$, $\tilde{N}\colon \Sigma\to\Img\mathbb{H}\cap\Spn(1)$, 
we obtain the maps $\alpha=\Phi_+^{-1}(N)$ and $\beta=\Phi^{-1}_-(\tilde{N})$. 
Let $a$ and $b$ be local maps  such that $a^\flat=\alpha$ and $(b^{-1})^\sharp=\beta$. 
Then $N=-aia^{-1}$ and $\tilde{N}=-bib^{-1}$. 
In addition, $\phi(a,b)$ defines the map $I^\Sigma\in\Upsilon_f$ such that $df\circ J_\Sigma=-I^\Sigma\,df$. 
\end{proof}

We connect the canonical factorization with the Weierstrass representation 
by Pedit and Pinkall (\cite{PP98}, Theorem 4.3) and 
obtain a global representation of a differential of a conformal map. 

Let $f\colon \Sigma\to V$ be a conformal map with left canonical lift $(f,\alpha)$ and right canonical lift $(f,\beta)$. 
Assume that $df=ak\,\eta\,b^{-1}$ is a canonical factorization. 
Let $L$ and $\tilde{L}$ be the trivial right quaternionic line bundles over $\Sigma$ with fiber $V$. 
Define a real bilinear pairling $(\enskip,\enskip)\colon L\otimes_{\mathbb{R}} \tilde{L}\to T^\ast\Sigma \otimes_{\mathbb{R}} V$ by 
\begin{gather*}
(v_0\lambda ,v_0\mu )=\overline{\lambda}\,df\,\mu\enskip 
(\lambda,\,\mu\in\mathbb{H}),
\end{gather*}
where $\otimes_{\mathbb{R}}$ indicates the tensor product over $\mathbb{R}$. 
The quaternionic linear complex structures $J_L$ and $J_{\tilde{L}}$ exist for $L$ and $\tilde{L}$, respectively, such that 
\begin{align*}
(v_0,v_0)\circ J_\Sigma
&=df\circ J_\Sigma 
=-\Phi_+(\alpha)\,df=df\,\Phi_-(\beta)
=-\Phi_+(\alpha)\,(v_0,v_0) \\
&=(v_0,v_0)\,\Phi_-(\beta) 
=(v_0\Phi_+(\alpha),v_0)
=(v_0,v_0\Phi_-(\beta)) \\
&=(J_Lv_0,v_0)=(v_0,J_{\tilde{L}}v_0).
\end{align*}

Define the quaternionic holomorphic structures $D_L$ and $D_{\tilde{L}}$ 
for $L$ and $\tilde{L}$, respectively, by

\begin{gather*}
D_L(v_0\lambda)=v_0\frac{1}{2}(d\lambda+\Phi_+(\alpha)\,d\lambda\circ J_\Sigma),\\
D_{\tilde{L}}(v_0\mu)=v_0\frac{1}{2}(d\mu+ \Phi_-(\beta)\,d\mu\circ J_\Sigma)\\
(\lambda,\,\mu\colon\Sigma\to \mathbb{H}).
\end{gather*}

Then, 
\begin{align*}
d(v_0\lambda,v_0\mu) =& \frac{1}{2}(d\overline{\lambda}-\,d\overline{\lambda}\circ J_\Sigma\, \Phi_+(\alpha))\wedge (v_0,v_0)\mu\\
                      &-(v_0\lambda,v_0)\wedge \frac{1}{2}(d\mu+ \Phi_-(\beta)\,d\mu\circ J_\Sigma)
\end{align*}
Then, 
for any nowhere-vanishing holomorphic section $v_0\lambda$ of $L$ and $v_0\mu$ of $\tilde{L}$, the pairing $(v_0\lambda,v_0\mu)$ is a closed one-form 
such that 
\begin{gather*}
(v_0\lambda,v_0\mu)\circ J_\Sigma=-\Phi_+(\tilde{\alpha})\,(v_0\lambda,v_0\mu)=
(v_0\lambda,v_0\mu)\,\Phi_-(\tilde{\beta}),\\
\tilde{\alpha}=\left(\frac{\overline{\lambda}a}{|\lambda|}\right)^\flat,\enskip \tilde{\beta}=\left(\frac{\mu b^{-1}}{|\mu|}\right)^{\sharp}.
\end{gather*}
If $(v_0\lambda,v_0\mu)$ is exact, then a conformal map 
$g\colon \Sigma\to V$ exists with canonical lifts $(g,\tilde{\alpha})$ and $(g,\tilde{\beta})$: 
\begin{gather*}
dg=(v_0\lambda,v_0\mu),\enskip 
dg\circ J_\Sigma=-\Phi_+(\tilde{\alpha})\,dg=
dg\,\Phi_-(\tilde{\beta}). 
\end{gather*}
The branch points of $g$ are the branch points of $f$. 

Let $E_L=\{\psi\in L:J_L\psi=\psi i\}$ and let $E_{\tilde{L}}=\{\psi\in \tilde{L}:J_{\tilde{L}}\psi=\psi i\}$. 
The bundle $E_L$ and $E_{\tilde{L}}$ are the eigenbundles of $J_L$ and $J_{\tilde{L}}$, respectively. 
\begin{lemma}\label{lem:cfw}
Let $f\colon\Sigma\to V$ be a conformal map with 
the canonical factorization $df=ak\,\eta\,b^{-1}$. 
Then, $v_0a$ and $v_0b$ are sections of $E_L$ and $E_{\tilde{L}}$ respectively.  
\end{lemma}
\begin{proof}
Because $df=ak\,\eta\,b^{-1}$, we obtain  
\begin{gather*}
\eta=-ka^{-1}\,df\,b=-ka^{-1}\,(v_0,v_0)\,b. 
\end{gather*}
Then, 
\begin{align*}
\eta\circ J_\Sigma 
&=-k(J_L(v_0a),v_0)b=-ka^{-1}(v_0,J_{\tilde{L}}(v_0b)) \\
&=-k(-i)(v_0a,v_0)b=-ka^{-1}(v_0,v_0b)i. 
\end{align*}
Thus, $J_L(v_0a)=v_0ai$ and $J_{\tilde{L}}(v_0b)=v_0bi$.  
\end{proof}

If the image of $f$ is contained in $V_c^\perp$, then 
we obtain $L=\tilde{L}$, $J_L=J_{\tilde{L}}$ and $a=b$. 
The eigenbundle $E$ is a spinor bundle of $\Sigma$. 

We assume that $\Sigma$ is a simply-connected open subset of $\mathbb{C}$. 
We recall procedures to construct a conformal map by 
two given conformal maps. 
Assume that $f\colon\Sigma\to V$ is a nowhere-vanishing conformal map 
and $g\colon\Sigma\to V$ is a conformal map. 
Pedit and Pinkall showed in \cite{PP98} that, 
if $df\circ J_\Sigma=N\,df$ and $dg\circ J_\Sigma=N\,dg$, 
then the map $h=f^{-1}g$ is a conformal map with 
$dh\circ J_\Sigma=f^{-1}Nf\,dh$, and 
if $df\circ J_\Sigma=-df\,\tilde{N}$ and $dg\circ J_\Sigma=-dg\,\tilde{N}$, 
then the map $h=gf^{-1}$ is a conformal map with 
$dh\circ J_\Sigma=-dh\,fNf^{-1}$. 
This result is used to construct a (Hamiltonian stationary) Lagrangian surface in 
\cite{Moriya08}. 
In terms of the canonical lift, we obtain the following lemma:  
\begin{lemma}\label{lem:quotient}
If the left canonical lift of $f$ is $(f,a^\flat)$ and the left canonical lift of 
$g$ is $(g,a^\flat)$, 
then the map $h=f^{-1}g$ is a conformal map with left canonical lift 
$(h,(|f|f^{-1}a)^\flat)$. 
If the right canonical lift of $f$ is $(f,(b^{-1})^\sharp)$ and 
the right canonical lift of $g$ is $(g,(b^{-1})^\sharp)$, then 
the map $h=gf^{-1}$ is a conformal map with 
right canonical lift $(h,((fb/|f|)^{-1})^\sharp)$.
\end{lemma} 
\begin{proof}
If the left canonical lift of $f$ is $(f,a^\flat)$, then $df\circ J_\Sigma=-aia^{-1}\,df$.  
The differential of $h=f^{-1}g$ is 
\begin{gather*}
dh=f^{-1}\,(-df\,f^{-1}g+dg). 
\end{gather*}
Thus,  
\begin{gather*}
dh\circ J_\Sigma=-f^{-1}aia^{-1}f(f^{-1}\,(-df\,f^{-1}g+dg)). 
\end{gather*}
Therefore, 
the map $h=f^{-1}g$ is a conformal map with left canonical lift $(h,(|f|f^{-1}a)^\flat)$. 

If the right canonical lift of $f$ is $(f,(b^{-1})^\sharp)$, then $df\circ J_\Sigma=df\,bib^{-1}$. 
The differential of $h=gf^{-1}$ is 
\begin{gather*}
dh=(dg-gf^{-1}\,df)f^{-1}. 
\end{gather*}
Thus, 
\begin{gather*}
dh\circ J_\Sigma=((dg-gf^{-1}\,df)f^{-1})fbib^{-1}f^{-1}. 
\end{gather*}
Therefore, 
the map $h=gf^{-1}$ is a conformal map with right canonical lift $(h,((fb/|f|)^{-1})^\sharp)$. 
\end{proof}

We cite the following lemma, which is subsequently applied. 
\begin{lemma}[\cite{BFLPP02}]\label{lem:wedge}
Let $\omega$ be a one-form with values in $V$ such that 
$\omega\circ J_\Sigma=N\,\omega=-\omega\,\tilde{N}$ for maps 
$N$, $\tilde{N}\colon \Sigma\to \Img\mathbb{H}\cap \Spn(1)$. 

If $\eta$ is a one-form with values in $V$ such that $\eta\circ J_\Sigma=\eta\,N$, then $\eta\wedge\omega=0$.  
If $\eta$ is a one-form with values in $V$ such that $\eta\circ J_\Sigma=-\tilde{N}\,\eta$, then $\omega\wedge \eta=0$. 

Assume that $\omega$ is nowhere vanishing. 
If $\eta$ is a one-form with values in $V$ such that $\eta\wedge\omega=0$, then $\eta\circ J_\Sigma=\eta\,N$.  
If $\eta$ is a one-form with values in $V$ such that $\omega\wedge \eta=0$, then $\eta\circ J_\Sigma=-\tilde{N}\,\eta$. 
\end{lemma}

We translate Lemma \ref{lem:wedge} into the language of conformal maps and their
canonical factorization. 
\begin{lemma}\label{lem:wedgecf}
Let $f\colon\Sigma\to V$ be a conformal map with left canonical lift $(f,a^\flat)$ and 
right canonical lift $(f,(b^{-1})^\sharp)$. 

If $h_L\colon\Sigma\to V$ is a conformal map 
with right canonical lift $(h_L,((ak)^{-1})^\sharp)$, 
then $dh_L\wedge df=0$.  
If $h_R\colon\Sigma\to V$ is a conformal map 
with left canonical lift $(h_R,(bk)^\flat)$, 
then $df\wedge dh_R=0$. 

Assume that $f$ is an immersion. 
If $h_L\colon\Sigma\to V$ is a map such that $dh_L\wedge df=0$, then $h_L$
is a conformal map with right canonical lift $(h_L,((ak)^{-1})^\sharp)$. 
If $h_R\colon\Sigma\to V$ is a map such that $df\wedge dh_R=0$, then 
$h_R$ is a conformal map with left canonical lift $(h_R,(bk)^\flat)$. 
\end{lemma}
\begin{proof}
Because $df\circ J_\Sigma=(-aia^{-1})\,df=df\,(bib^{-1})$, 
Lemma \ref{lem:wedgecf} is based on Lemma \ref{lem:wedge}. 
\end{proof}

\begin{definition}\label{def:Bt}
We refer to the conformal map $h_R\colon\Sigma\to V$ with $df\wedge dh_R=0$ as the \textit{right B\"{a}cklund transform} of $f$ and 
a conformal map $dh_L\colon\Sigma\to V$ with $dh_L\wedge df=0$ as the \textit{left B\"{a}cklund transform} of $f$. 
\end{definition}
In \cite{BFLPP02}, the right B\"{a}cklund transform and the left B\"{a}cklund transform 
for a Willmore surface are given and called the forward B\"{a}cklund transform and the backward B\"{a}cklund transform respevtively.
The B\"{a}cklund transforms for a conformal map of a Riemann surface to $S^4$ is defined in \cite{LP05}. 
Restricting the codomain of a conformal map to $S^4$ with one point removed and fixing the stereographic projection from the point, 
the B\"{a}cklund transforms are reduced to Definition \ref{def:Bt} (see 
\cite{Moriya13}). 
The Darboux transforms of a conformal map of a Riemann surface into $S^4$ is defined in \cite{BLPP12}. 
In a similar manner, as the B\"{a}cklund transforms, 
we obtain a Darboux transform of a conformal map of a Riemann surface into $\mathbb{E}^4$ 
(see \cite{Moriya13}). 
In terms of canonical lifts, a Darboux transform is explained as follows. 
\begin{lemma}\label{lem:Dt}
Let $f\colon\Sigma\to V$ be a conformal map with left canonical lift $(f,a^\flat)$ and 
right canonical lift $(f,(b^{-1})^\sharp)$. 

If $h_L\colon \Sigma\to V$ is a left B\"{a}cklund transform of $f$ with $dg_L=h_L\,df$  
and nowhere vanishing, then 
the map $\widehat{f}_L:=-h_L^{-1}g_L+f\colon\Sigma\to V$ is a conformal map with 
right canonical lift $(\widehat{f}_L,((akh_L^{-1}g_L)^{-1}/|h_L^{-1}g_L|)^\sharp)$. 

If $h_R\colon \Sigma\to V$ is a right B\"{a}cklund transform of $f$ with $dg_R=df\,h_R$  
and nowhere vanishing, then 
the map $\widehat{f_R}:=-g_Rh_R^{-1}+f\colon\Sigma\to V$ is a conformal map with left canonical lift 
$(\widehat{f}_R,(g_Rh_R^{-1}bk/|g_Rh_R^{-1}|)^\flat)$. 
\end{lemma}
\begin{proof}
Because 
\begin{align*}
d\widehat{f}_L &=-dh_L^{-1}g_L-h_L^{-1}\,dg_L+df=-dh_L^{-1}g_L,\\
d\widehat{f}_R &=-dg_R\,h_R^{-1}-g_R\,dh_R^{-1}+df=-g_R\,dh_R^{-1},
\end{align*}
the lemma holds. 
\end{proof}
\begin{definition}\label{def:Dt}
We refer to $\widehat{f}_L$ in Lemma \ref{lem:Dt} as the \textit{left Darboux transform} of $f$ by a left B\"{a}cklund transform $h_L$ and 
$\widehat{f}_R$ as the \textit{right Darboux transform} of $f$ by 
a right B\"{a}cklund transform $h_R$. 
\end{definition}

We obtain the following relation between the area of a conformal map 
and the area of its Darboux transform.  
\begin{theorem}\label{thm:aBDt}
Let $f\colon \Sigma\to V$ be a conformal map, 
let $h_L$ be the right B\"{a}cklund transform of $f$, 
let $\widehat{f_L}$ be the right Darboux transform by $h_L$, 
let $h_R$ the right B\"{a}cklund transform of $f$ and 
let $\widehat{f_R}$ be the right Darboux transform by $h_R$. 
Assume that $dg_L=h_L\,df$ and $dg_R=df\,h_R$. 
If $f$, $\widehat{f}_L$, $df$, $d\widehat{f}_L$, 
$d(h_L^{-1}g_L)$ and $d(g_Rh_R^{-1})$ are square integrable,  
then 
\begin{gather*}
A(f)+A(\widehat{f}_L)-\langle\!\langle df,d\widehat{f}_L\rangle\!\rangle_\Sigma
=\frac{\|d(h_L^{-1}g_L)\|^2_\Sigma}{2},\\
A(f)+A(\widehat{f}_R)-\langle\!\langle df,d\widehat{f}_R\rangle\!\rangle_\Sigma
=\frac{\|d(g_Rh_R^{-1})\|^2_\Sigma}{2}. 
\end{gather*}
\end{theorem}
\begin{proof}
By the definition of the left Darboux transform, we obtain 
\begin{gather*}
d(f-\widehat{f}_L)=df-d\widehat{f}_L=d(h_L^{-1}g_L). 
\end{gather*}
Thus, 
\begin{gather*}
\|df-d\widehat{f}_L\|^2_\Sigma=\|d(h_L^{-1}g_L)\|^2_\Sigma.
\end{gather*}
Then,  
\begin{gather*}
2A(f)+2A(\widehat{f}_L)-2\langle\!\langle df,d\widehat{f}_L\rangle\!\rangle_\Sigma
=\|d(h_L^{-1}g_L)\|^2_\Sigma. 
\end{gather*}
Then, we obtain the former equality. 
In a similar manner, we have the latter equality. 
\end{proof}

\section{Super-conformal maps}
We apply the canonical factorization to super-conformal maps. 
In our canonical factorization, 
the $(1,0)$-form explains the intrinsic Riemannian geometry of a conformal map and 
the maps into Sp(1) give the generalized Gauss map of a surface. 
Moreover these are expressed in terms of the multiplication of $\mathbb{H}$. 
We obtain an estimate of the area of a super-conformal map. 

Prior to our discussion of super-conformal maps, 
we investigate the map $N=-aia^{-1}\colon \Sigma \to S^2=\Img\mathbb{H}\cap \Spn(1)$ with $a\colon\Sigma\to\Spn(1)$. 
We consider $S^2=\Img\mathbb{H}\cap \Spn(1)$ to be the Riemann sphere $\mathbb{C}P^1$. 
Let $w$ the stereographic projection from the north pole. 
Then 
\begin{gather*}
w\mapsto\frac{2\real w}{|w|^2+1}i+\frac{2\Img w}{|w|^2+1}j+\frac{|w|^2-1}{|w|^2+1}k
\end{gather*}
is a holomorphic parametrization of $S^2\setminus\{k\}$. 
The following lemma is proven in \cite{Moriya15}. We provide an
alternate short proof. 
\begin{lemma}\label{lem:holN}
The map $N\colon\Sigma\to \Img\mathbb{H}\cap \Spn(1)\cong\mathbb{C}P^1$ is holomorphic if and only if $dN\circ J_\Sigma =-N\,dN=dN\,N$. 

The map $N\colon\Sigma\to \Img\mathbb{H}\cap \Spn(1)\cong\mathbb{C}P^1$ is anti-holomorphic if and only if $dN\circ J_\Sigma =N\,dN=-dN\,N$. 
\end{lemma}
\begin{proof}
Let $(x,y)$ be a local conformal coordinate of $\Sigma$ with 
\begin{gather*}
J_\Sigma\frac{\partial}{\partial x}=\frac{\partial}{\partial y}. 
\end{gather*}
The map $N\colon \Sigma\to \Img\mathbb{H}\cap \Spn(1)$ is holomorphic if and only if 
the vector product  $N\times N_x$ is equal to $-N_y$. 
Differentiating $N^2=-1$, we obtain $N\,d N=-(d N)N$. 
Then,  
$N$ is holomorphic if and only if $d N\circ J_\Sigma=-N\,d N=(d N)N$. 
Similarly, $N$ is anti-holomorphic if and only if $d N\circ J_\Sigma=N\,d N=-(d N)N$. 
\end{proof}
This lemma is translated as follows: 
\begin{lemma}\label{lem:hola}
Let $\alpha\colon \Sigma\to \Spn(1)/\Un(1)$ be a map and 
let $a\colon\Sigma\to\Spn(1)$ be a map with $a^\flat=\alpha$. 
Let $N=-\Phi_+(\alpha)$. 
The map $N\colon\Sigma\to \Img\mathbb{H}\cap \Spn(1)\cong\mathbb{C}P^1$ is holomorphic if and only if the map $\alpha\colon\Sigma\to\Spn(1)/\Un(1)\cong\mathbb{P}(V_+)$ is anti-holomorphic. 
The map $N\colon\Sigma\to \Img\mathbb{H}\cap \Spn(1)\cong\mathbb{C}P^1$ is anti-holomorphic if and only if the map $\alpha\colon\Sigma\to\Spn(1)/\Un(1)\cong\mathbb{P}(V_+)$ is holomorphic. 

Let $\beta\colon\Sigma\to\Spn(1)/\Un(1)$ be a map and 
$b\colon\Sigma\to \Spn(1)$ be a map with $(b^{-1})^\sharp=\beta$. 
Let $\tilde{N}=-\Phi_-(\beta)$.  
A map $\tilde{N}\colon\Sigma\to \Img\mathbb{H}\cap \Spn(1)\cong\mathbb{C}P^1$ is holomorphic if and only if $\beta\colon\Sigma\to\Spn(1)/\Un(1)\cong\mathbb{P}(V_-)$ is anti-holomorphic. 
A map $\tilde{N}\colon\Sigma\to \Img\mathbb{H}\cap \Spn(1)\cong\mathbb{C}P^1$ is anti-holomorphic if and only if $\beta\colon\Sigma\to\Spn(1)/\Un(1)\cong\mathbb{P}(V_-)$ is holomorphic. 
\end{lemma}
A variant of this lemma is also proven in \cite{Moriya15}. 
We provide an improved proof by the form $N=-aia^{-1}$.  
\begin{proof}
The map $a^\flat$ is holomorphic 
if and only if 
the map $c\colon \Sigma\to \mathbb{C}\setminus\{0\}$ exists such that the map 
$av_0c \colon \Sigma\to \mathbb{P}(V_+)$ is holomorphic. 
Thus, $a^\flat$ 
is holomorphic if and only if 
the map $c\colon \Sigma\to \mathbb{C}\setminus\{0\}$ exists such that the map 
$ac \colon \Sigma\to V_+$ is holomorphic. 
The map 
$ac \colon \Sigma\to V_+$ is holomorphic if and only if 
$d(ac)\circ J_\Sigma=d(ac)\,i$.  
The differential of $N=-aia$ is 
\begin{gather*}
dN=d(-(ac)i(ac)^{-1})=(ac)(-(ac)^{-1}\,d(ac)\,i+i(ac)^{-1}\,d(ac))(ac)^{-1}.
\end{gather*}
Thus, if $ac$ is holomorphic, then $dN\circ J_\Sigma=-dN\,N=N\,dN$. 
If $a^\flat$ is holomorphic, then $N$ is anti-holomorphic. 
If $N$ is anti-holomorphic, then $dN\circ J_\Sigma=N\,dN=-dN\,N$. 
Thus,
\begin{align*}
&(-(ac)^{-1}\,d(ac)\,i+i(ac)^{-1}\,d(ac))\circ J_\Sigma\\
=&(-(ac)^{-1}\,d(ac)\,i+i(ac)^{-1}\,d(ac))i\\
=&-i(-(ac)^{-1}\,d(ac)\,i+i(ac)^{-1}\,d(ac)). 
\end{align*}
Therefore, we can select $c$ such that $ac$ is holomorphic. 
Then, $a^\flat$ is holomorphic. 
Similarly, $a^\flat$ is anti-holomorphic if and only if $N$ is holomorphic. 

The map $(b^{-1})^\sharp$ is holomorphic 
if and only if 
the map $c\colon \Sigma\to \mathbb{C}\setminus\{0\}$ exists such that a map 
$c^{-1}b^{-1}=(bc)^{-1} \colon \Sigma\to V_-$ is holomorphic. 
The map 
$(bc)^{-1} \colon \Sigma\to V_-$ is holomorphic if and only if 
$d(bc)^{-1}=-i\,d(bc)^{-1}$. 
The differential of $\tilde{N}=-bib^{-1}$ is 
\begin{gather*}
d\tilde{N}=d(-(bc)i(bc)^{-1})=(bc)(d(bc)^{-1}\,(bc)i-i\,d(bc)^{-1}\,(bc))(bc)^{-1}.
\end{gather*}
Thus, if $(bc)^{-1}$ is holomorphic, then $d\tilde{N}\circ J_\Sigma=-d\tilde{N}\,\tilde{N}=\tilde{N}\,d\tilde{N}$. 
Therefore, if $(b^{-1})^\sharp$ is holomorphic, then $\tilde{N}$ is anti-holomorphic. 
If $\tilde{N}$ is anti-holomorphic, then $d\tilde{N}\circ J_\Sigma=\tilde{N}\,d\tilde{N}=-d\tilde{N}\,\tilde{N}$. 
Thus, 
\begin{align*}
&(d(bc)^{-1}\,(bc)i-i\,d(bc)^{-1}\,(bc))\circ J_\Sigma\\
=&(d(bc)^{-1}\,(bc)i-i\,d(bc)^{-1}\,(bc))i\\
=&-i(d(bc)^{-1}\,(bc)i-i\,d(bc)^{-1}\,(bc)). 
\end{align*}
Therefore, we can choose $c$ such that $(bc)^{-1}$ is holomorphic. 
Then, $(b^{-1})^\sharp$ is holomorphic. 
\end{proof}

A conformal map is referred to as a \textit{super-conformal} map if 
its curvature ellipse is a circle at each immersed point. 
As shown in \cite{Moriya09}, a super-conformal map is a B\"{a}cklund transform of a minimal surface.
A holomorphic function is a super-conformal map. 
Let $f\colon \Sigma\to V$ be a conformal map with $d f\circ J_\Sigma=N\,df=-df\,\tilde{N}$.  
The curvatures of $f$ are calculated by $N$ and $\tilde{N}$ in \cite{BFLPP02}. 
We observe that the following lemma holds from Section 8.2 in \cite{BFLPP02}.
\begin{lemma}\label{lem:scN}
A conformal map $f$ is super-conformal if and only if $N$ or $\tilde{N}$ is anti-holomorphic. 
\end{lemma}

If $f$ is super-conformal, then 
a holomorphic lift of $f$ to the twistor space exists (see, for example,  \cite{BFLPP02}, Theorem 5). 
We have distinguished this holomorphic lift. 
\begin{theorem}\label{schl}
The left canonical lift or the right canonical lift of 
a conformal map is holomorphic if and only if the conformal map is 
super-conformal. 
\end{theorem}
\begin{proof}
If $f$ is conformal, then $f$ is always holomorphic with respect to $I^\Sigma$. 
If the left canonical lift $(f,\alpha)$ is holomorphic, then 
$\alpha$ is holomorphic. 
If $(f,\alpha)$ is holomorphic, then $f$ is super-conformal. 
Similarly, if the right canonical left $(f,\beta)$ is 
holomorphic, then $f$ is super-conformal. 

Conversely, if $f$ is super-conformal with $df\circ J_\Sigma=N\,df=-df\,\tilde{N}$, then $N$ or $\tilde{N}$ is anti-holomorphic by Lemma \ref{lem:scN}. If $N=-\Phi_+(\alpha)$, then 
$\alpha$ is holomorphic by Lemma \ref{lem:hola}. 
Similarly, 
if $\tilde{N}=-\Phi_-(\beta)$, then 
$\beta$ is holomorphic by Lemma \ref{lem:hola}. 
\end{proof}

Let $f\colon \Sigma\to V$ be a super-conformal map with 
canonical factorization $df=ak\,\eta\,b^{-1}$. 
We may assume that $a\colon\Sigma\to V_+$ is holomorphic. 
Then, the local complex function $c$ exists such that 
$d(ac)\circ J_\Sigma=d(ac)\,i$. 
We obtain the factorization $df=\tilde{a} \zeta$ with 
$\tilde{a}=ac$ and 
$\zeta=k\bar{c}^{-1}\,\eta\,b^{-1}$. 
Differentiating $df=\tilde{a} \zeta$, we obtain  
\begin{gather*}
0=d(df)=d\tilde{a} \wedge \zeta+ \tilde{a} \, d\zeta. 
\end{gather*}
The branch points of $f$ is exactly the zeros of $ac$ or the zeros of 
$\zeta$. 
We employ this factorization for an estimate of the area. 
Let $D =\{z\in \mathbb{C}:|z|<1\}$, 
and $D_r =\{z\in \mathbb{C}:|z|<r\}$. 
We recall the Schwarz lemma:
\begin{theorem}[\cite{Schwarz90}, \cite{GK06}]\label{thm:Schwarz}
Let $f\colon D\to D$ be a holomorphic function such that $f(0)=0$. 
Then, $|f(z)|\leq |z|$ on $D$ and $|f_z(0)|\leq 1$. 
The equality holds if and only if $|f_z(0)|=1$ or there exists $z_0\in D\setminus\{0\}$ such that $|f(z_0)|=|z_0|$. 
\end{theorem}

We have the following area estimate for a super-conformal map by 
the factorization and the Schwarz lemma.

\begin{theorem}\label{thm:areas}
Let $f\colon D\to V$ be a super-conformal map that is branched at $0$.  
Assume that $f$ has the factorization 
$df=\tilde{a}\,\zeta$ by a one-form $\zeta$ and 
a holomorphic map
$\tilde{a}\colon D\to V_+$. 
Let $a_0$ and $a_1$ be holomorphic functions such that 
$\tilde{a}=a_0+ka_1$. 
Assume that $0$ is a zero of $a_0$ and $a_1$ of order 
$m_0-1$ and $m_1-1$ respectively. 
Assume that positive numbers $C_{a_0}$, $C_{a_1}$ and $C_\zeta$ 
exist such that 
$|a_0(z)/z^{m_0-2}|\leq C_{a_0}$, $|a_1(z)/z^{m_1-2}|\leq C_{a_1}$ and $\zeta\wedge(\overline{\zeta}\circ J_\Sigma)\geq C_\zeta\,dz\wedge (d\overline{z}\circ J_\Sigma)$ on $D$. 
Then  
\begin{gather*}
A(f|_{D_r})\leq \pi C_\zeta\left(\frac{C_{a_0}^2}{m_0}r^{2m_0}+\frac{C_{a_1}^2}{m_1}r^{2m_1}\right)\enskip (0<r< 1). 
\end{gather*}
Assume that $\zeta\wedge(\overline{\zeta}\circ J_\Sigma)= C_\zeta\,dz\wedge (d\overline{z}\circ J_\Sigma)$ and 
$z_0\in D\setminus\{0\}$ exists such that 
\begin{itemize}
\item $|a_0(z_0)|=C_{a_0}|z_0|^{m_0-1}$ or $|((a_0)_z/z^{m_0-2})_z(0)|=C_{a_0}$.
\item $|a_1(z_0)|=C_{a_1}|z_0|^{m_1-1}$ or $|((a_1)_z/z^{m_1-2})_z(0)|=C_{a_1}$. 
\end{itemize}
Then, equality holds.  
\end{theorem}
\begin{proof}
By the Schwarz lemma, we obtain  
$|a_0(z)|\leq C_{a_0}|z|^{m_0-1}$ and $|a_1(z)|\leq C_{a_1}|z|^{m_1-1}$. 
The equality simultaneously holds if and only if $z_0\in D\setminus\{0\}$ exists such that the following equalities hold:
\begin{itemize}
\item $|a_0(z_0)|=C_{a_0}|z_0|^{m_0-1}$ or $|((a_0)_z/z^{m_0-2})_z(0)|=C_{a_0}$.
\item $|a_1(z_0)|=C_{a_1}|z_0|^{m_1-1}$ or $|((a_1)_z/z^{m_1-2})_z(0)|=C_{a_1}$. 
\end{itemize}
By the Schwarz inequality, the area of $f|_{D_r}$ is 
\begin{align*}
A(f|_{D_r})&=-\frac{1}{2}\int_{D_r}df\wedge(d\overline{f}\circ J_\Sigma)
=-\frac{1}{2}\int_{D_r}|\tilde{a}|^2\,\zeta\wedge(\overline{\zeta}\circ J_\Sigma)\\
&\leq -\frac{1}{2}C_\zeta\int_{D_r}(|a_0|^2+|a_1|^2)\,dz\wedge (d\overline{z}\circ J_\Sigma)\\
&\leq -\frac{1}{2}C_\zeta\int_{D_r}(C_{a_0}^2|z|^{2m_0-2}+C_{a_1}^2|z|^{2m_1-2})\,dz\wedge (d\overline{z}\circ J_\Sigma)\\
&=\pi C_\zeta\left(\frac{C_{a_0}^2}{m_0}r^{2m_0}+\frac{C_{a_1}^2}{m_1}r^{2m_1}\right).
\end{align*}
Because $a$ is holomorphic, 
the condition for equality is based on  
the condition for equality in the Schwarz lemma. 
\end{proof}

\end{document}